\documentclass[11pt,a4paper]{article}%
\usepackage[centertags]{amsmath}
\usepackage{amsfonts}
\usepackage{amssymb}
\usepackage{amsthm}
\usepackage{epsfig}
\usepackage{setspace}
\usepackage{ae}
\usepackage{eucal}
\usepackage[usenames]{color}%
\usepackage{graphicx}

\theoremstyle{plain}
\newtheorem{thm}{Theorem}[section]
\newtheorem{lem}[thm]{Lemma}
\newtheorem{cor}[thm]{Corollary}
\newtheorem{prop}[thm]{Proposition}
\theoremstyle{definition}
\newtheorem{defi}[thm]{Definition}

\theoremstyle{remark}

\newtheorem{rem}[thm]{Remark}

\begin{document}


\title{Density estimates for differential equations of second order satisfying a weak H\"{o}rmander condition} 

\author{J\"org Kampen }
\maketitle





\begin{abstract} 
We prove an extension of H\"{o}rmander's classical result on hypoelliptic second order equations in \cite{H} where the coefficients of the related vector fields are globally Lipschitz and satisfy the classical H\"{o}rmander condition on a dense set while the density still exists in a classical sense. Furthermore, H\"{o}rmanders classical result and related density estimates in \cite{KS} based on Malliavin calculus are recovered from an analytical point of view.
\end{abstract}





\section{Introduction}
H\"{o}rmander's paper on hypoelliptic equations of second order shows that there are smooth densities if a Lie-algebra condition is pointwise satisfied, which is determined by the defining vector fields of the equation. This condition is also called the 'H\"{o}rmander condition'. Let us recall the situation. For positive  natural numbers $m,n$ consider a matrix-valued function 
\begin{equation}\label{vcoeff}
x\rightarrow (v_{ji})^{n,m}(x),~1\leq j\leq n,~0\leq i\leq m,
\end{equation}
 on ${\mathbb R}^n$, and $m$ smooth vector fields 
\begin{equation}\label{vvec}
V_i=\sum_{j=1}^n v_{ji}(x)\frac{\partial}{\partial x_j},
\end{equation}
where $0\leq i\leq m$. These vector fields define a
Cauchy problem on $[0,\infty)\times {\mathbb R}^n$ for the distribution $p$ of the form
\begin{equation}
	\label{hoer1}
	\left\lbrace \begin{array}{ll}
		\frac{\partial p}{\partial t}=\frac{1}{2}\sum_{i=1}^mV_i^2p+V_0p\\
		\\
		p(0,x;y)=\delta_y(x),
	\end{array}\right.
\end{equation}
where $\delta_y(x)=\delta(x-y)$ is the Dirac delta distribution with an argument shifted by the vector $y\in {\mathbb R}^n$. H\"{o}rmander showed that the operator in (\ref{hoer1}) is hypoelliptic on $(0,\infty)\times {\mathbb R}^n$ in the sense that for all $y\in {\mathbb R}^n$ the distribution $p$ satisfying (\ref{hoer1}) is smooth on the domain $(0,\infty)\times {\mathbb R}^n$, if the following condition is satisfied:
for all $x\in {\mathbb R}^n$  assume that
\begin{equation}
H_x={\mathbb R}^n,
\end{equation}
where
\begin{equation}\label{Hoergenx}
\begin{array}{ll}
H_x:=\mbox{span}{\Big\{} V_i(x), \left[V_j,V_k \right](x),
\left[ \left[V_j,V_k \right], V_l\right](x),\\
\\
\cdots |1\leq i\leq m,~0\leq j,k,l,\cdots \leq m {\Big \}},
\end{array}
\end{equation}
where $\left[.,.\right]$ denotes the Lie bracket of vector fields as usual. Usually this goes with the assumption that
the coefficients of the vector fields are smooth (i.e. $C^{\infty}$) and bounded  with bounded derivatives, i.e. $v_{ji} \in C_{b}^{\infty}\left({\mathbb R}^n \right)$. Sometimes linear growth for the functions $v_{ji}$ themselves is allowed ( $v_{ji} \in C_{b,l}^{\infty}\left({\mathbb R}^n \right)$ in symbols). Indeed this makes no real difference. These strong regularity assumptions are due to the fact that iterative application of the Lie bracket operator involves derivatives of arbitrary order of the coefficients of the vector fields $V_i$. Another reason is that lower regularity of the coefficients implies lower regularity of a distributional solution in general. However, the proof of existence of a density (of possibly lower regularity) does not require assumptions that are that strong. The H\"{o}rmander condition implies that the diffusion matrix of second order coefficients and the vector field of first order coefficients interact via diffusion on subspaces induced by the second order coefficients, and shifts and rotations induced by the first order coefficients. Locally this implies the existence of a density in a Trotter-product limit. In order to preserve these phenomena it is sufficient to require a H\"{o}rmander condition on a dense subset of ${\mathbb R}^n$ - as we shall see in detail below. This leads us to 'weak' formulations for the H\"{o}rmander condition. The price to pay is lower regularity of the density, of course, but we can recover the original result in \cite{H} and \cite{KS} from a different point of view if we allow for the classical H\"{o}rmander condition.
\begin{defi}
The vector fields $V_i=\sum_{j=1}^n v_{ji}(x)\frac{\partial}{\partial x_j},~1\leq j\leq n,~0\leq i\leq m$ defined on a domain $\Omega\subseteq {\mathbb R}^n$ satisfy the weak H\"{o}rmander condition if for all $1\leq j\leq n,~0\leq i\leq m$ the function $v_{ij}$ is globally Lipschitz and there is a dense set $D\subset \Omega$ (dense with respect to the standard topology on Euclidean space) such that for all $x\in D$ the classical H\"{o}rmander condition holds at $x$, i.e., $H_x={\mathbb R}^n$ for all $x\in D$.
\end{defi}
We prove the following extension of H\"{o}rmander's theorem
\begin{thm}
	\label{kampen}
Assume that $V_i=\sum_{j=1}^n v_{ji}(x)\frac{\partial}{\partial x_j}$
where $0\leq i\leq m,~1\leq j\leq n$ satisfy the weak H\"{o}rmander condition.  	 
Then the density $p$ exists in $C^{1,2}$, i.e. 
\begin{equation}
\begin{array}{ll}
p:(0,T]\times {\mathbb R}^n\times{\mathbb R}^n\rightarrow {\mathbb R}\in C^{1,2}\left( (0,T]\times {\mathbb R}^n\times{\mathbb R}^n\right). 
\end{array}
\end{equation}
Moreover, if the classical H\"{o}rmander condition holds, then for each nonnegative natural number $j$, and multiindices $\alpha,\beta$ there are increasing functions of time
\begin{equation}\label{constAB1}
A_{j,\alpha,\beta}, B_{j,\alpha,\beta}:[0,T]\rightarrow {\mathbb R},
\end{equation}
and functions
\begin{equation}\label{constmn1}
n_{j,\alpha,\beta}, 
m_{j,\alpha,\beta}:
{\mathbb N}\times {\mathbb N}^d\times {\mathbb N}^d\rightarrow {\mathbb N},
\end{equation}
such that 
\begin{equation}\label{pxest1}
\begin{array}{ll}
{\Bigg |}\frac{\partial^j}{\partial t^j} \frac{\partial^{|\alpha|}}{\partial x^{\alpha}} \frac{\partial^{|\beta|}}{\partial y^{\beta}}p(t,x,y){\Bigg |}\\
\\
\leq \frac{A_{j,\alpha,\beta}(t)(1+x)^{m_{j,\alpha,\beta}}}{t^{n_{j,\alpha,\beta}}}\exp\left(-B_{j,\alpha,\beta}(t)\frac{(x-y)^2}{t}\right)
\end{array}
\end{equation}
In the case where the weak H\"{o}rmander condition holds the relations (\ref{pxest1}) hold for $j\leq 1$ and $|\alpha|+|\beta|\leq 2$, where $|\alpha|=\sum_{i=1}^n\alpha_i$ denotes the order of a multiindex $\alpha=(\alpha_1,\cdots,\alpha_n)$ with nonnegative entries, and similarly for the multiindex $\beta$.
Moreover, the properties of all functions (\ref{constAB1}) and  (\ref{constmn1}) depend in general on the level of iteration of Lie-bracket iteration at which the H\"{o}rmander condition becomes true.
\end{thm}

H\"{o}rmander's classical theorem 
essentially tells us that the density $p$, i.e. the solution of (\ref{hoer1}) satisfies $p\in C^{\infty}
\left(\left(0,\infty\right) \times {\mathbb R}^n \right)$ as a function of $t$ and $x$ (given $y$ as a parameter). One of the main motivations of the Malliavin calculus was to examine this result from a probabilistic point of view in the case of a classical H\"{o}rmander condition. For diffusions this research culminated in the work of  Kusuoka and Stroock (cf. \cite{KS}) who proved the following result.

\begin{thm}
	\label{stroock}
	Consider a $d$-dimensional difffusion process of the form
	\begin{equation}\label{stochm}
		\mathrm{d}X_t \ = \ \sum_{i=1}^d\sigma_{0i}(X_t)\mathrm{d}t + \sum_{j=1}^{d}\sigma_{ij}(X_t)\mathrm{d}W^j_t
	\end{equation}
with $X(0)=x\in {\mathbb R}^d$ with values in ${\mathbb R}^d$ and on a time interval $[0,T]$.
Assume that $\sigma_{0i},\sigma_{ij}\in C^{\infty}_{lb}$. Then the law of the process $X$ is absolutely continuous with respect to the Lebesgue measure, and the density $p$ exists and is smooth, i.e. 
\begin{equation}
\begin{array}{ll}
p:(0,T]\times {\mathbb R}^d\times{\mathbb R}^d\rightarrow {\mathbb R}\in C^{\infty}\left( (0,T]\times {\mathbb R}^d\times{\mathbb R}^d\right). 
\end{array}
\end{equation}
Moreover, for each nonnegative natural number $j$, and multiindices $\alpha,\beta$ there are increasing functions of time
\begin{equation}\label{constAB}
A_{j,\alpha,\beta}, B_{j,\alpha,\beta}:[0,T]\rightarrow {\mathbb R},
\end{equation}
and functions
\begin{equation}\label{constmn}
n_{j,\alpha,\beta}, 
m_{j,\alpha,\beta}:
{\mathbb N}\times {\mathbb N}^d\times {\mathbb N}^d\rightarrow {\mathbb N},
\end{equation}
such that 
\begin{equation}\label{pxest}
\begin{array}{ll}
{\Bigg |}\frac{\partial^j}{\partial t^j} \frac{\partial^{|\alpha|}}{\partial x^{\alpha}} \frac{\partial^{|\beta|}}{\partial y^{\beta}}p(t,x,y){\Bigg |}\\
\\
\leq \frac{A_{j,\alpha,\beta}(t)(1+x)^{m_{j,\alpha,\beta}}}{t^{n_{j,\alpha,\beta}}}\exp\left(-B_{j,\alpha,\beta}(t)\frac{(x-y)^2}{t}\right)
\end{array}
\end{equation}
Moreover, all functions (\ref{constAB}) and  (\ref{constmn}) depend on the level of iteration of Lie-bracket iteration at which the H\"{o}rmander condition becomes true.
\end{thm}

The theorem in (\ref{stroock})  is also sometimes formulated in a probabilistic manner. We note
\begin{cor}
	In the situation of (\ref{stroock}) above, solution $X_t^x$ starting at $x$ is in the standard Malliavin space $D^{\infty}$, and there are constants $C_{l,q}$ depending on the derivatives of the drift and dispersion coefficients such that for some constant $\gamma_{l,q}$
\begin{equation}\label{xprocessest}
|X_t^x|_{l,q}\leq C_{l,q}(1+|x|)^{\gamma_{l,q}}.
\end{equation}
Here $|.|_{l,q}$ denotes the norm where derivatives up to order $l$ are in $L^q$ (in the Malliavin sense).
\end{cor}
Before we turn to our different point of view let us dwell a little on this probabilistic side.
The main phenomenon behind this statement,  is that random walks around infinitesimal cubes of dimension $n$ of diffusions of H\"{o}rmander type produce drifts which involve iterated Lie-algebra commutators of the vector fields which in turn define parts of first order coefficients of the diffusion. Essentially you may consider a Skorohod-diffusion
\begin{equation}\label{diffhoer}
dX_{t}=V_1(X_t)d\circ W_1+V_2(X_t)d \circ W_2,
\end{equation}
and consider a path of the Brownian motion which describes a square of  infinitesimal length $\delta:=\sqrt{\delta t}$, lets say in a calculus which has explicit infinitesimals in its r\'{e}pertoire
(the functional analytic calculus of Connes or the nonstandard calculus of Robinson are examples (cf. \cite{K})). Then following this path of $(W_1,W_2)$ one time around the square we get 
\begin{itemize}
\item[i)] Taylor expansion:
\begin{equation}
 X_{t_1}=x+V_1(x)\delta+DV_1(x)V_1(x)\frac{\delta }{2}+\mbox{h.o.}
\end{equation}
\item[ii)] Taylor expansion second step in direction of second coord.:
\begin{equation}
\begin{array}{ll}
X_{t_2}=X_{t_1}+V_2(X_{t_1})\delta +DV_2(X_{t_1})V_2(X_{t_1})\frac{\delta^2}{2}+\mbox{h.o.}\\
\\
=x+V_1(x)\delta +DV_1(x)V_1(x)\frac{\delta^2}{2}\\
\\
+V_2(x)\delta +DV_2(x)V_1(x)\delta^2+DV_2(x)V_2(x)\frac{\delta^2}{2}+\mbox{h.o.}
\end{array}
\end{equation}
\item[iii)] Taylor expansion second step in direction of second coord.:
\begin{equation}
\begin{array}{ll}
X_{t_3}=X_{t_2}-V_1(X_{t_2})\delta +DV_1(X_{t_2})V_1(X_{t_2})\frac{\delta^2}{2}+\mbox{h.o.}\\
\\
=x+V_2(x)\delta +[V_2,V_1](x)\delta^2+\mbox{other terms}
\end{array}
\end{equation}
\item[iv)] Taylor expansion second step in direction of second coord.:
\begin{equation}
\begin{array}{ll}
X_{t_4}=X_{t_3}-V_2(X_{t_3})\delta +DV_2(X_{t_3})V_2(X_{t_3})\frac{\delta^2}{2}+\mbox{h.o.}\\
\\
=x+[V_2,V_1](x)\delta^2+\mbox{h.o.}
\end{array}
\end{equation}
\end{itemize}
Iteration of this procedure leads to iterated Lie-brackets which define the drift of the corresponding diffusions. This indicates that the diffusion in (\ref{diffhoer}) fills the whole space defined by the H\"{o}rmander condition, and this suggests that a density should exist which is continuous with respect to the Lebegues measure.
Next to the extension mentioned in this paper we provide an elementary analytic proof of the result in \cite{KS}. Sticking to a probabilistic point of view a little longer we may define the H\"{o}rmander condition (more precisely) recursively. W may define
\begin{equation}
H_x:=\cup_{k=0}^{\infty} H^k_x,
\end{equation}
where
\begin{equation}
H^0_x=\mbox{span}\left\lbrace \sigma_1(x),\cdots ,\sigma_m(x)\right\rbrace ,
\end{equation}
($\sigma_i$ being the columns of the diffusion matrix $\sigma$), and $H^{k+1}_x$ is defined recursively in terms of $H^k_x$ by
\begin{equation}
H^{k+1}_x:=\left\lbrace \left[\sigma_0,f \right]_x, \left[\sigma_1,f \right]_x,\cdots ,\left[\sigma_m,f \right]_x|f\in H^k_x\right\rbrace .
\end{equation}
Here, for $f,g\in C^1\left({\mathbb R}^n,{\mathbb R}^n\right) $
\begin{equation}
\left[f,g \right]_x:=f(x)\nabla g(x)-g(x)\nabla f(x) 
\end{equation}
In this case we also say that $H_x$ is generated by the functions $\sigma_0,\sigma_1,\cdots ,\sigma_m$ and we write
\begin{equation}
H_x=H_x\left[\sigma_0,\sigma_1,\cdots ,\sigma_m \right]. 
\end{equation}
This is indeed an equivalent statement of the H\"{o}rmander condition. 
There are several observations here which lead us to a weak formulation. The strong regularity conditions of the H\"{o}rmander condition are imposed by the need of derivatives of arbitrary order in order to formulate the condition. However, the proof by Malliavin calculus shows that the essential condition for existence of a density is the $L^p$-invertibility of the Malliavin covariance matrix for $p\geq 1$. Below we shall see that this condition has its natural analytic counter part if we look at extension of classical constructions of the fundamental solution. Since the H\"{o}rmader condition is formulated as a bunch of local pointwise conditions, it is natural to ask if weak forms of this conditions are available where $L^p$-invertability for all $p\geq 1$ of the Malliavin covariance matrix or an analytic counter part of this condition is still preserved. A weak formulation or distributional formulation of the H\"{o}rmander condition leads  nowhere, since the condition looses its pointwise sense and becomes rather meaningless for higher order derivatives. Since the H\"{o}rmander condition is formulated pointwise it makes sense to look for a subset of points of ${\mathbb R}^n$ where it should hold in order that a density exists. We may loose hypoellipticity, i.e., some regularity. Two features of parabolic equations of second order seem fitting. Constructions of densities by weakly singular integrals as in classical methods naturally require Lipschitz continuity and $L^p$ invertibilty of matrices related to the diffusion matrix. Further, we may approximate functions with lower regularity, e.g. with Lipschitz regularity by functions which are smooth pointwise. Furthermore, diffusion equations are stable in the sense that a small perturbation of the coefficients leads to small perturbation of the solution.
We often have a situation where coefficients are globally only Lipschitz but are $C^{\infty}$ almost everywhere. Think of a multivariate stochastic volatility models in mathematical finance, for example. Anyway, it turns out that a weak H\"{o}rmander condition is sufficient in order to construct classical densities. Hypoellipticity is lost in general, but that is another matter.
 
From a probabilistic point of view we may equivalently say that the weak H\"{o}lder condition is satisfied for a diffusion with coefficients $v_{ij}$ as in (\ref{vcoeff}) and vector fields as in \ref{vvec} if
\begin{equation}
D:=\left\lbrace x| H_x\left[\sigma_0,\sigma_1,\cdots ,\sigma_m \right]={\mathbb R}^n\right\rbrace 
\end{equation}
is dense in ${\mathbb R}^n$, i.e. if $\overline{D}={\mathbb R}^n$, where the bar-superscript denotes again the closure with respect to the standard topology, and where $H_x$ is defined as in (\ref{Hoergenx}) above.
We note that if the weak H\"{o}rmander contion holds on $ K$, then for all $x$ in ${\mathbb R}^n$ there sequence of functions $f^k_i\in C^{\infty}\left({\mathbb R}^n,{\mathbb R}^n\right), 0\leq i\leq m $ such that for all $0\leq i\leq n$
\begin{equation}
\lim_{k\uparrow \infty}\sup_{y\in K}|\sigma_i(y)-f^k_i(y)|=0
\end{equation}
and where for each $k$ we have 
\begin{equation}
{\mathbb R}^n=H_x\left[f_0^k,\cdots,f_m^k \right]. 
\end{equation}
Depending on the conditions on the coefficients we get less or more regularity.
In the next section we consider two essential examples which illustrate the meaning of the H\"{o}rmander's result and our way to prove it. In the final section of this paper we prove the extension of H\"{o}rmander's result.  

\section{Two essential examples (analysis of the H\"{o}rmander condition)}

Consider the equation
\begin{equation}\label{ex1}
\frac{\partial u}{\partial t}-\lambda_2\frac{\partial^2 u}{\partial x_2^2}+x_2\mu_1\frac{\partial u}{\partial x_1}=0
\end{equation}
for some constants $\lambda_2,\mu_1>0$. Since the coefficients are linear the H\"{o}rmander space can be determined after one iteration of the Lie-bracket recursion. We have
\begin{equation}
H:=\mbox{span}\left\lbrace (0,\lambda_2)^T,(\lambda_2\mu_1,0\right\rbrace={\mathbb R}^2=H_x 
\end{equation}
independently of the argument $x$. H\"{o}rmander's result tells us that there is a density. It shows this in any case for bounded domains, but linear growth of the coefficient functions is sufficient to have the result also for unbounded domains.
This example shows that the diffusion can be degenerated to a linear subspace while we still have a density. The reason is that we have a smoothing diffusion on a subspace $\left\lbrace  x_1=0\right\rbrace $ while the drift term implies some rotation. Let us have a closer look at this. Consider a rotation $D=(d_{ij})$, i.e., $2\times 2$ matrix with $DD^T=1$, and the related coordinate transformation $y=Dx$ ($x=D^{-1}y$), where for $x=(x_1,x_2)$ and $y=(y_1,y_2)$ we have
\begin{equation}
\begin{array}{ll}
y_1=d_{11}x_1+d_{12}x_2,~y_2=d_{21}x_1+d_{22}x_2.
\end{array}
\end{equation}
We get for $u(t,x)=v(t,y)$
\begin{equation}
\begin{array}{ll}
\frac{\partial u}{\partial x_1}=\frac{\partial v}{\partial y_1}d_{11}+\frac{\partial v}{\partial y_2}d_{21}\\
\\
\frac{\partial u}{\partial x_2}=\frac{\partial v}{\partial y_1}d_{12}+\frac{\partial v}{\partial y_2}d_{22},
\end{array}
\end{equation}
and
\begin{equation}
\frac{\partial^2 u}{\partial x_2^2}=d_{12}^2\frac{\partial^2 v}{\partial y_2^2}+2d_{22}d_{12}\frac{\partial^2 v}{\partial y_1\partial y_2}d_{22}d_{12}+d^2_{22}\frac{\partial^2 v}{\partial y_2^2}.
\end{equation}
Hence
\begin{equation}
\begin{array}{ll}
\frac{\partial v}{\partial t}-\lambda_2\left(d_{12}^2\frac{\partial^2 v}{\partial y_1^2}+2d_{22}d_{12}\frac{\partial^2 v}{\partial y_1\partial y_2}+d^2_{22}\frac{\partial^2 v}{\partial y_2^2}\right)\\
\\
 +\left(D^{-1}y\right)_2 \mu_1\left(\frac{\partial v}{\partial y_1}d_{11}+\frac{\partial v}{\partial y_2}d_{21}\right) =0,
\end{array}
\end{equation}
where $\left(D^{-1}y\right)_2$ denotes the second component of $D^{-1}y$. As we move through all rotations we observe that for $\lambda_2>0$ the symbol of the rotated operator 
\begin{equation}
\lambda_2\left(d_{12}^2\xi^2_1+2d_{22}d_{12}\xi_1\xi_2+d^2_{22}\xi_2^2\right) 
\end{equation}
is elliptic, and for for almost all rotations $D$ it degenerates only at an exceptional set of points. It is only for $d_{12}=0$ or $d_{22}=0$ (rendering $d_{22}=1$ and $d_{12}=1$) that the degeneracy occurs on a whole line $\xi_2=0$ or, respectively, $\xi_1=0$. 
Now we may solve this equation locally by iteratively solving the vector field equation for a function $v^f$ with
\begin{equation}\label{vecf}
\frac{\partial v^{f}}{\partial t} +\left(D^{-1}y\right)_2 \mu_1\left( \frac{\partial v^{f}}{\partial y_1}d_{11}+\frac{\partial v^{f}}{\partial y_2}d_{21}\right) =0
\end{equation}
and the diffusion equation for a function $v^d$ with
\begin{equation}\label{diff}
\frac{\partial v^{d}}{\partial t}-\lambda_2\left(d_{12}^2\frac{\partial^2 v^{d}}{\partial y_2^2}+2d_{22}d_{12}\frac{\partial^2 v^d}{\partial y_1\partial y_2}+d^2_{22}\frac{\partial^2 v^d}{\partial y_2^2}\right) =0,
\end{equation}
where we may use a Trotter product formula. The vector field locally moves data via a flow along the characteristic of the first order equation. Now as the first order term in (\ref{ex1}) depends on the variable $x_2$ for noncritical valued of the solution of (\ref{vecf}) we have a smoothing effect of the diffusion in both directions. We can look at this also this way: start with the point $(x_1,x_2)$ and apply the flow $\cal{F}^{e_1}_t$ of the vector field equation (\ref{vecf}). After small time the initial point is transported to $(\cal{F}^{e_1}x,1,x_2)$, where ${\cal F}^{e_1}x_1$ depends on $x_2$ (except for some exceptional set of critical points). Then the diffusion part $\frac{\partial u}{\partial t}-\lambda_2\frac{\partial^2 u}{\partial x_2^2}$ of the operator has it smoothing effect not only on the variable $x_2$ but also on ${\cal F}^{e_1}x_1$  which depends effectively on $x_2$ in another small time step. The associated Trotter product formula then proves that this effect is preserved for Trotter-type iterations in the limit. This 'phenomenological background' of the classical H\"{o}rmander theorem which makes this intuition precise (we shall make it precise in our own way below).
Our second essential example is about second order degenerate parabolic equations which satisfy a  reduced H\"{o}rmander condition.  Theses are second order equations without a 'drift vector field' which satisfy the H\"{o}rmander condition. More precisely, if the problem
\begin{equation}
	\label{hoer2*}
	\left\lbrace \begin{array}{ll}
		\frac{\partial q}{\partial t}=\frac{1}{2}\sum_{i=1}^mV_i^2q+\tilde{V}_0q,\\
		\\
		q(0,x;y)=\delta_y(x),
	\end{array}\right.
\end{equation}
is equivalent to
\begin{equation}
\label{hoerdiff}
	\left\lbrace \begin{array}{ll}
		\frac{\partial q}{\partial t}=\frac{1}{2}\sum_{i,j=1}^n\left( \sigma\sigma^T\right) _{ij}\frac{\partial^2}{\partial x_i\partial x_j}q,\\
		\\
		q(0,x;y)=\delta_y(x),
	\end{array}\right.
\end{equation}
and the vector fields $\tilde{V}_0,V_1,\cdots ,V_m$ satisfy the (weak) H\"{o}rmander condition, then we say that a reduced H\"{o}rmander condition is satisfied. Starting with a problem as in (\ref{hoer1}) we can move to a reduced problem (\ref{hoer2*}) by a 'drift shift' $-V_0+\tilde{V}_0$. The reduced H\"{o}rmander condition is then expressed via the associated recursion for $x\in {\mathbb R}^n$ with  
\begin{equation}
H^{0,{\tiny \mbox{red}}}_x=H^0_x=\mbox{span}\left\lbrace \tilde{V}_0(x),V_1(x),\cdots ,V_m(x)\right\rbrace ,
\end{equation}
 and for $k+1\geq 1$ with $H^{k,\mbox{red}}_x$ by
\begin{equation}
H^{k+1,{\tiny \mbox{red}}}_x:=\left\lbrace \left[V_1,f \right]_x, \left[V_2,f \right]_x,\cdots ,\left[V_m,f \right]_x|f\in H^k_x\right\rbrace .
\end{equation}
We say that a reduced H\"{o}rmander condition is satisfied at $x\in {\mathbb R}^n$ if
\begin{equation}
H^{{\tiny \mbox{red}}}_x=\cup_{k\geq 0}H^{k,\mbox{red}}_x={\mathbb R}^n.
\end{equation}
Furthermore we say that the reduced H\"{o}rmander condition is satisfied in a domain $U\subseteq {\mathbb R}^n$ if
\begin{equation}
H^{{\tiny \mbox{red}}}_x={\mathbb R}^n \mbox{ for all }x\in U.
\end{equation}
For these vector fields we can observe that the classical H\"{o}rmander condition implies that the points of non-degeneracy have to be dense. In the case of a second order parabolic equation which satisfies a reduced H\"{o}rmander condition it is natural to write down a Levy expansion formally and then to look which additional conditions have to be satisfied in order that the expansion converges to a classical solution. You observe that the second order coefficient matrix should be invertible and the higher order terms in the formal Levy expansion suggest that integrals over integer powers of the inverse of the second order coefficient matrix should also be invertible. This is quite close to the assumed $L^p$-invertability of the Malliavin covariance matrix for all $p\geq 1$ which is the usual condition in the context of stochastic analysis. It is the first step in our proof to make this statement precise and show that the weak H\"{o}rmander condition together with a certain $L^p$ invertibility for $p\geq 1$ is indeed sufficient in order to if the weak H\"{o}rmander condition is satisfied.    

\section{Proof of main theorem 1.3}
First we consider diffusions without drift terms, and where the reduced H\"{o}rmander condition holds. This condition is much stronger than the classical H\"{o}rmander condition where the support of the diffusion coefficients may be a located in a much smaller subspace of ${\mathbb R}^n$. In the case of the stronger reduced H\"{o}rmander condition the elliptic degeneracies are of Lebesgue measure zero, while this is not true for the full classical H\"{o}rmander condition. 
First we prove that in the case of a reduced H\"{o}rmander condition we find for each argument $x\in {\mathbb R}^n$ a positive small number $\epsilon >0$ and a ball $B_{\epsilon}(x)$ of radius $\epsilon >0$ around $x$ such that a density $(t,x,y)\rightarrow p(t,x;y)$ exists on $(0,\infty)\times B_{\epsilon}(x)\times B_{\epsilon}(x)$. We do this by a Levy type expansion using the regularity of of the diffusion coefficients, and where the leading term is modified. Using this representation of local densities we also obtain local Gaussian estimates for the densities and their derivatives, i.e., a local form of the estimate (\ref{pxest1}) above in the case of the reduced H\"{o}rmander condition.

Then we extend this proof of local existence of densities in the general case of diffusions with drift term which satisfy the classical H\"{o}rmander condition. In this step the expansive representations of the density of the reduced case do not hold in general. However, around each $x$ we find a ball where in a lower dimensional subspace we can obtain a similar expansive local representation of the density. The first order terms then have the effect of a infinitesimal rotation such that the diffusion becomes effective in all space direction. We make these observations precise and construct a density by a local Trotter product formula. Local estimates of the form (\ref{pxest1}) are then obtained following this construction. In a third step we design a scheme which extends local densities to global densities in small time. Using the semigroup property  is then rather straightforward to conclude that there exists a density which is global in time. Following the whole construction we can then establish regularity and upper bounds in the sense of Kusuoka-Stroock in the case of the classical H\"{o}rmander condition straightforwardly. Finally we show that in the case of the weak H\"{o}rmander condition there is still a classical density, i.e., a density in $C^{1,2}\left((0,\infty)\times {\mathbb R}^n \right)$. In the case of the stronger form of the weak H\"{o}rmander condition we show that the Gaussian estimates of the form (\ref{pxest1}) are still valid up to second order derivatives in space and first order derivatives in time.  Next we implement this program.

\subsection{Existence of local densities in the reduced case}

If the reduced H\"{o}rmander condition holds, then first we proof that in any neighborhood $B_{\epsilon}(x)$ of a point $x\in {\mathbb R}^n$ we can find a point at which the second order coefficient matrix of a reduced H\"{o}rmander diffusion is invertible, and such that the elliptic degeneracies, i.e., the set of points where $\mbox{det}\left(a_{ij}\right)$ is zero is itself of Lebesgues measure zero. Recall that we speak of a reduced H\"{o}rmander diffusion if the reduced H\"{o}rmander condition is satisfied. The idea is to use this observation in order to construct a leading term of a generalised local Levy expansion.
From the diffusion equation above which satisfies the H\"{o}mander condition at each point $x\in {\mathbb R}^n$ we have weak ellipticity in the sense that
\begin{equation}
(a_{ij})=\sigma\sigma^T\geq 0
\end{equation}
At each point we can diagonalize the (w.l.o.g.) symmetric matrix $(a_{ij})=\sigma\sigma^T$ such that we get
a function
\begin{equation}
x\rightarrow \Lambda(x)=D(x)\sigma\sigma^T(x)D^T(x)
\end{equation}
where $\Lambda(x)=\mbox{diag}(\lambda_1(y),\cdots ,\lambda_n(x))$ is a diagonal matrix. We have that $x\rightarrow \Lambda(x)\in C^{\infty}$. 
Next we have  
\begin{lem}\label{lem1}
Assume that the reduced H\"{o}rmander condition is satisfied and let $x\in {\mathbb R}^n$.
Let $\mbox{diag}(\lambda_i(y))$ the diagonal form of the symmetric second order coefficient matrix of the reduced H\"{o}rmander diffusion. Then for each $\epsilon >0$ and each ball $B_{\epsilon}(x)$ of radius $\epsilon>0$ around $x$ there is a $y\in B^{\epsilon}(x)$ such that
\begin{equation}
\lambda_i(y)>0 \mbox{ for all }1\leq i\leq n.
\end{equation}
\end{lem}
\begin{proof}
As $\lambda_i(y)\geq 0$ for all $y\in B_{\epsilon}(x)$ by weak ellipticity every point where $\lambda_i(y)=0$ is critical such that we can apply Sard's theorem to
\begin{equation}
y\rightarrow \mbox{det}(a_{ij})=\Pi_{i=1}^n\lambda_i(y)
\end{equation}
since we have sufficient regularity of the functions $\lambda_i(y)$. However, since we assume this regularity only on a dense set later, here is a more elementary consideration.
Assume for the contrary that there is an index $i$ and a ball $B_{\epsilon}(x)$ such that 
$\lambda_i(y)=0$ for all $y\in B_{\epsilon}(x)$. The matrices $\mbox{diag}(\lambda_i(y))$ we have on nonzero entry in each column such that the Lie bracket operations of the H\"{o}rmander condition simplify. In order to produce a nonzero entry at $y$ after one Lie bracket iteration requires
\begin{equation}
0\neq -\frac{\partial \lambda_i(y)}{\partial x_j}\lambda_j(y)
\end{equation}
for some $j\neq i$, a contradiction to the assumption. This means that a nonzero entry will not appear in the $i$ column for iterations of the Lie bracket operation as well 
\end{proof}
Sard's theorem together with weak ellipticity and the H\"{o}rmander condition have another immediate consequence
\begin{lem}\label{lem2}
Assume that the reduced H\"{o}rmander condition is satisfied. Then the zeros of the map
\begin{equation}
y\rightarrow \mbox{det}\left(a_{ij}(y)\right)=\Pi_{i=1}^n\lambda_i(y) 
\end{equation}
are of Lebegues measure zero.
\end{lem}
Now let 
\begin{equation}
A_0:=\left\lbrace x\in {\mathbb R}^n| \mbox{det}\left(a_{ij}(y)\right)=0\right\rbrace 
\end{equation}
be set of elliptic degeneracies of the diffusion matrix, and let
\begin{equation}
\delta:=\left\lbrace (x,y)|x,y\in {\mathbb R}^n~\&~x=y \right\rbrace
\end{equation}
be the set of points on the diagonal.
The two observation stated in lemma \ref{lem1} and lemma \ref{lem2} lead us to a candidate of a leading term of a local Levy type expansion of the form
\begin{equation}\label{Ndef}
\begin{array}{ll}
N(t,x;s,y):=\\
\\
\left\lbrace \begin{array}{ll}\frac{1}{\sqrt{4\pi \Pi_{i=1}^n\lambda_i(y) (t-s)}^n}\exp\left(-\sum_{i=1}^n\frac{\lambda^{-1}_{i}(y)(x_i-y_i)^2}{4 (t-s)} \right)
            ~\mbox{ if }y\not\in A_0\mbox{ or }(x,y)\not\in \delta\\
	    \\
	    0~\mbox{if }y\in A_0\mbox{ and }(x,y)\in \delta,
           \end{array}\right.
	   \end{array}
\end{equation}
where we understand that in the first alternative in the definition (\ref{Ndef}) we have for $y\in A_0\mbox{ and }x\not\in \delta$ that 
\begin{equation}
\frac{1}{\sqrt{4\pi \Pi_{i=1}^n\lambda_i(y) (t-s)}^n}\exp\left(-\sum_{i=1}^n\frac{\lambda^{-1}_{i}(y)(x_i-y_i)^2}{4 (t-s)} \right){\Big |}_{y\in A_0}=0,
\end{equation}
which is the natural limit. Note that this function has well-defined multivariate spatial and time derivatives of any order for $t>s$ and $x,y\in {\mathbb R}^n$.

Having observed this we can prove a local form of the main part of theorem 1.3 for problems without first order terms where a reduced H\"{o}rmander condition holds, i.e., we prove
\begin{lem}
Consider a class of Cauchy problems 
\begin{equation}
	\label{hoer2}
	\left\lbrace \begin{array}{ll}
		\frac{\partial q}{\partial t}=\sum_{i=1}^mV_i^2q+V_0q\\
		\\
		q(0,x;y)=\delta_y(x),
	\end{array}\right.
\end{equation}
which is equivalent to a pure diffusion problem
\begin{equation}
\label{hoerdiff0}
	\left\lbrace \begin{array}{ll}
		0=\frac{\partial q}{\partial t}-\sum_{i,j=1}^na_{ij}\frac{\partial^2}{\partial x_i\partial x_j}q\equiv L_0y\\
		\\
		q(0,x;y)=\delta_y(x),
	\end{array}\right.
\end{equation}
and assume that the classical H\"{o}rmander condition is satisfied for all $x\in {\mathbb R}^n$, i.e.,
\begin{equation}
H^{\tiny {\mbox{red}}}_x\left[V_1,\cdots ,V_n \right] ={\mathbb R}^n \mbox{ for all }x\in {\mathbb R}^n.
\end{equation}
Then there exists a density $p$ which is a solution of (\ref{hoerdiff0}), and such that for all $s\in [0,\infty)$ and $y\in {\mathbb R}^n$ with $t>s$ we have $(t,x)\rightarrow p(t,x;s,y)\in C^{1,2}\left(\left(0,\infty\right)\times {\mathbb R}^n  \right)$ and the Kusuoka-Stroock estimates are satisfied.
\end{lem}
In order to prove this we reconsider the basic classical idea of constructing solutions for equations of type (\ref{hoerdiff0}). However, classically, such expansions are done for operators with a strictly elliptic spatial part. We next observe that we can extend the method to operators where elliptic degeneracies are of Lebegues measure zero. As we observed we have Lebegues measure zero for the set of degeneracies $A_0$ for diffusions, where a reduced H\"{o}rmander condition holds. The method easily extends to equations of with drift terms as well if a reduced H\"{o}rmander condition holds. However, in case of a full H\"{o}rmander condition the set of elliptic degeneracies is not of measure zero such that we cannot apply this method immediately. However, we shall see that we can define more involved leading terms and use the results which we obtain in the case of diffusions without drift. o let us consider the case (\ref{hoerdiff0}) first. We shall consider the additional effects of a nonzero drift term and a general H\"{o}rmander condition later. In order to solve a rather general equation like (\ref{hoerdiff0}) you may first solve a simplified approximative equation and the look for correction terms. As a simplified approximative equation it is natural to consider the equation 
\begin{equation}
\label{hoerdiffy}
	\left\lbrace \begin{array}{ll}
		\frac{\partial q}{\partial t}=\sum_{i,j=1}^na_{ij}(y)\frac{\partial^2}{\partial x_i\partial x_j}q\\
		\\
		q(0,x;y)=\delta_y(x),
	\end{array}\right.
\end{equation}
i.e., the equation with constant second order coefficients $a_{ij}$ evaluated at the parameter $y\in {\mathbb R}^n$.  For fixed $y\in {\mathbb R}^n\setminus A_0$ the equation (\ref{hoerdiffy}) has the explicit solution
\begin{equation}\label{ngendef}
N_0(t,x;s,y)=\frac{1}{\sqrt{4\pi \mbox{det}(a_{ij}) (\tau-s)}^n}\exp\left(-\sum_{i,j=1}^n\frac{a^{{\tiny \mbox{inv}}}_{ij}(y)(x_i-y_i)(x_j-y_j)}{4 (t-s)} \right),
\end{equation}
where the coefficients $a^{{\tiny\mbox{inv}}}_{ij}(y)$ denote the entries of the inverse of the matrix $(a_{ij}(y))$, i.e.,
\begin{equation}
\left( a^{{\tiny\mbox{inv}}}_{ij}(y)\right):=\left(a_{ij}(y)\right)^{-1}.
\end{equation}
We introduced a subscript zero in (\ref{ngendef}) in order to indicate that this functions equals $N$, but is written in given coordinates while the coordinates of the definition of $N_0$ were transformed.
Note that we did not solve equation (\ref{hoerdiffy}) for $y\in A_0$ as we defined $N=0$ for all $x,y$ with $y\in A_0~\&x=y$, as the initial condition becomes zero in this case. However, as $A_0$ is of Lebegues measure zero and the expansive representations are integral representations this 'almost-everywhere-solution' of (\ref{hoerdiffy}) leads to globally valid expansions of densities.   
Now if we plug in this solution of the natural approximative equation into the original equation we obtain
\begin{equation}
\label{hoerdiff0step1}
\begin{array}{ll}
		\frac{\partial N_0}{\partial t}-\sum_{i,j=1}^na_{ij}(x)\frac{\partial^2}{\partial x_i\partial x_j}N_0\\
		\\
		=\sum_0-\sum_{i,j=1}^na_{ij}(x)\frac{\partial^2}{\partial x_i\partial x_j}N_0\\
		\\
		=\sum_{i,j=1}^n\left( a_{ij}(y)-a_{ij}(x)\right) \frac{\partial^2}{\partial x_i\partial x_j}N_0\\
		\\
		=\frac{\partial N_0}{\partial t}-a_{ij}(x)\frac{\partial^2}{\partial x_i\partial x_j}N_0=:L_0N_0.
		\end{array}
\end{equation}
 The right side represents the deficiency of the approximative solution. We may treat it as a source term for a better approximative equation, where we notice that it is not easier to solve (\ref{hoerdiff0}) than to solve the original equation. We have to approximate again. Well $L_0N_0$ measures the deficiency of the first a approximative solution as well, and this observation leads us to the natural next better approximative equation of (\ref{hoerdiff}) of the form 
\begin{equation}
\label{hoerdiff0step2}
\begin{array}{ll}
		\frac{\partial N^1}{\partial t}-\sum_{i,j=1}^na_{ij}(y)\frac{\partial^2}{\partial x_i\partial x_j}N^1\\
		\\
		=L_0N_0=\sum_{i,j=1}^n\left( a_{ij}(y)-a_{ij}(x)\right) \frac{\partial^2}{\partial x_i\partial x_j}N_0,
		\end{array}
\end{equation}
to be solved with Dirac data $\delta(x-y)$ for given parameter $y$. This is for $(x,y)\not\in \delta$. In general we define
\begin{equation}
\begin{array}{ll}
L_0N_0:=\left\lbrace \begin{array}{ll}\sum_{i,j=1}^n\left( a_{ij}(y)-a_{ij}(x)\right) \frac{\partial^2}{\partial x_i\partial x_j}N_0~\mbox{if}~y\not\in A_0\mbox{ or }(x,y)\not\in \delta\\
\\
0~\mbox{if }y\in A_0\mbox{ and }(x,y)\in \delta,
\end{array}\right.
\end{array}
\end{equation}
For all parameters $s\in [0;T],y\in {\mathbb R}^n$ and $s<t\in (0,T],x\in {\mathbb R}^n$ the solution has the representation
\begin{equation}
N^1(t,x;s,y)=N_0(t,x;s,y) +\int_s^t\int_{{\mathbb R}^n}N_0(t,x;r,z)L_0N_0(r,z;s,y)dzdr
\end{equation}
If we plug in this approximation into the original equation (\ref{hoerdiff}), then we get
\begin{equation}
\label{hoerdiffappr1}
\begin{array}{ll}
		\frac{\partial N^1}{\partial t}-\sum_{i,j=1}^na_{ij}(x)\frac{\partial^2}{\partial x_i\partial x_j}N^1\\
		\\
		=L_0N_0+\int_s^t\int_{{\mathbb R}^n}L_0N_0(.,.;r,z)L_0N_0(r,z;s,y)dzdr.
		\end{array}
\end{equation}
This leads to the next natural approximative equation
\begin{equation}
\label{hoerdiffappr12}
\begin{array}{ll}
		\frac{\partial N^2}{\partial t}-\sum_{i,j=1}^na_{ij}(y)\frac{\partial^2}{\partial x_i\partial x_j}N^2\\
		\\
		=L_0N+\int_s^t\int_{{\mathbb R}^n}L_0N(.,.;r,z)L_0N(r,z;s,y)dzdr.
		\end{array}
\end{equation}
This leads to a natural series $(N^k)_{k\in {\mathbb N}}$ of approximations, and we may expect that
\begin{equation}
p_0:=\lim_{k\uparrow \infty}N^k 
\end{equation}
is a solution of (\ref{hoerdiffy}) if
\begin{equation}
\delta N^k=N^{k}-N^{k-1}\downarrow 0
\end{equation}
in some appropriate strong norm.
This leads directly to the classical Levy expansion of the fundamental solution on an arbitrary domain $\Omega\subseteq {\mathbb R}^n$. We may write it down including first order terms and note
\begin{lem}
We assume that the reduced H\"{o}rmander condition is satisfied.
Assume first that $y\not\in A_0$ and $y\not\in \delta$.
 For the fundamental solution $p$ of
\begin{equation}
\label{hoerdiffstep3}
	\left\lbrace \begin{array}{ll}
		0=\frac{\partial q}{\partial t}-\sum_{i,j=1}^na_{ij}\frac{\partial^2}{\partial x_i\partial x_j}q+\sum_{j=1}^nb_j\frac{\partial q}{\partial x_j}\equiv Lq\\
		\\
		q(0,x;y)=\delta_y(x),
	\end{array}\right.
\end{equation}
we have the solution
\begin{equation}\label{pexp1}
p(t,x;s,y):=N_0(t,x;s,y)+\int_s^{\tau}\int_{\Omega}N_0(t,x;\sigma,\xi)\phi(\sigma,\xi;s,y)d\sigma d\xi,
\end{equation}
where $\phi$ is a recursively defined function, i.e.,
\begin{equation}\label{pexp2}
\phi(t,x;s,y)=\sum_{m=1}^{\infty}(LN_0)_m(t,x;s,y),
\end{equation}
along with the recursion
\begin{equation}\label{pexp3}
\begin{array}{ll}
(LN_0)_1(t,x;s,y)=LN_0(t,x;s,y)\\
\\
=\frac{\partial N}{\partial t}-\sum_{i,j=1}^n a_{ij}\frac{\partial^2}{\partial x_i\partial x_j}N_0 +\sum_{j=1}^n b_j\frac{\partial N_0}{\partial x_j}\\
\\
=\sum_{i,j=1}^n \left( a_{ij}(y)-a_{ij}(x)\right) \frac{\partial^2}{\partial x_i\partial x_j}N_0 +\sum_{j=1}^n b_j\frac{\partial N_0}{\partial x_j},\\
\\
(LN_0)_{m+1}(t ,x):=\int_s^t\int_{\Omega}\left( LN_0(t,x;\sigma,\xi)\right)_m LN_0(\sigma,\xi;s,y)d\sigma d\xi.
\end{array}
\end{equation} 
For $y\in A_0$ and $(x,y)\not\in\delta$ we set $p(t,x;s,y)=0$ and the same for $y\in A_0$ where $(x,y)\in\delta$. 
In the special case $b_i\equiv 0,~1\leq i\leq n$ this is the formal Levy expansion for (\ref{hoerdiff}) as well.
\end{lem}

\begin{proof}
Standard classical methods of proof of Gaussian estimates by weakly singular integral majorants can be extended to the case where the set of elliptic degeneracies is of measure zero. Another method of proof can be extracted by our stronger Gaussian estimates in the proof of the theorem below which leads to full regularity in the case of a reduced H\"{o}rmander condition.
\end{proof}

In the reduced case the estimates in \cite{KS} are now easy to obtain. We have
\begin{thm}
The estimates (\ref{pxest}) in theorem \ref{kampen} hold if the reduced H\"{o}rmander condition is satisfied. Moreover in this case they hold in a stronger form without the factor $(1+x)^{m_{j,\alpha,\beta}}$.
\end{thm}

\begin{proof}
We consider transformed coordinates and the Gaussian leading term in the form (\ref{Ndef}) above. First we consider multivariate spatial derivatives with respect to the variable $x$ of this leading term.
\begin{equation}\label{Ndef**}
\begin{array}{ll}
N(t,x;0,y):=\\
\\
\left\lbrace \begin{array}{ll}\frac{1}{\sqrt{4\pi \Pi_{i=1}^n\lambda_i(y) t}^n}\exp\left(-\sum_{i=1}^n\frac{\lambda^{-1}_{i}(y)(x_i-y_i)^2}{4 t} \right)
            ~\mbox{ if }y\not\in A_0\mbox{ or }(x,y)\not\in \delta\\
	    \\
	    0~\mbox{if }y\in A_0\mbox{ and }(x,y)\in \delta,
           \end{array}\right.
	   \end{array}
\end{equation}
If $y\not\in A_0\mbox{ and }(x,y)\not\in \delta$, then for $t>0$ we have
\begin{equation}
D^{\alpha}_xN_0(t,x;0,y)=p_{\alpha}\left(\frac{(x-y)}{t}\right)N_0(t,x;0,y)
\end{equation}
for some polynomial $p_{\alpha}$. Note that $p_{\alpha}$ may have an absolute terms $\frac{1}{t^p}$, which we may consider as parameters as $t>0$.  As  $y\not\in A_0$ we have for $z:=\frac{(x-y)}{2t}$
\begin{equation}
\begin{array}{ll}
{\big |}p_{\alpha}\left(\frac{(x-y)}{t}\right)N_0(t,x;0,y){\big |}
\leq {\Big |}\sup_{z\in {\mathbb R}}\left( p_{\alpha}(z)
\exp\left(-\frac{1}{2}\sum_{i=1}^n\lambda^{-1}_{i}(y)z_i^2 \right)\right)\times\\
\\ 
\frac{1}{\sqrt{4\pi \Pi_{i=1}^n\lambda_i(y) t}^n}\exp\left(-\frac{1}{2}\sum_{i=1}^n\lambda^{-1}_{i}(y)z_i^2 \right){\Big |}\\
\\
\leq \frac{C}{t^p}\frac{1}{\sqrt{4\pi \Pi_{i=1}^n\lambda_i(y) }}\exp\left(-\frac{1}{2}\sum_{i=1}^n\lambda^{-1}_{i}(y)\frac{(x_i-y_i)^2}{t} \right)
\end{array}
\end{equation}
for some $C>0$ and some nonnegative integer $p$. A similar estimate holds for multivariate derivatives with respect to the variables $y$, such that altogether in the case $y\not\in A_0\mbox{ and }(x,y)\not\in \delta$ we get
\begin{equation}
{\Big |}D^{\alpha}_xD^{\beta}_yN_0(t,x;0,y){\Big |}\leq \frac{C}{t^p}\frac{1}{\sqrt{4\pi \Pi_{i=1}^n\lambda_i(y) }^n}
\exp\left(-\frac{1}{2}\sum_{i=1}^n\lambda^{-1}_{i}(y)\frac{(x_i-y_i)^2}{t} \right)
\end{equation}
for some $C>0$ and some nonnegative integer $p$. 
Next note that for a Cauchy sequence $y_n\in {\mathbb R}^n\setminus A_0$ the supremum
\begin{equation}
 {\Bigg |}\sup_{z\in {\mathbb R}}\left( p_{\alpha}(z)
\exp\left(-\frac{1}{2}\sum_{i=1}^n\lambda^{-1}_{i}(y_n)z_i^2 \right)\right){\Bigg |}
\end{equation}
is uniformly bounded such that the estimates extend to the degeneracies taking into account the definition of $N_0$ at the elliptic degeneracies. As $t>0$ the estimates also extend to the points on the diagonal $\delta$.
Next as we apply time derivatives $D^m_t:=\frac{d^m}{dt^m}$ to multivariate spatial derivatives of the leading term $N_0$ we effectively have
\begin{equation}
D^m_tD^{\alpha}_xD^{\beta}_yN_0(t,x;0,y)=D^m_t\left( p_{\alpha\beta}\left(\frac{(x-y)}{t}\right)N_0(t,x;0,y)\right) 
\end{equation}
for some polynomial $p_{\alpha\beta}$. Applying Leibniz rule and the estimates for the spatial derivatives of $N_0$ it is then straightforword to get the estimate in (\ref{pexp1}) below.
Hence for the leading term of our expansion $N_0$ we have
\begin{equation}\label{pxestN}
\begin{array}{ll}
{\Bigg |}\frac{\partial^j}{\partial t^j} \frac{\partial^{|\alpha|}}{\partial x^{\alpha}} \frac{\partial^{|\beta|}}{\partial y^{\beta}}N_0(t,x;0,y){\Bigg |}\\
\\
\leq \frac{A_{j,\alpha,\beta}(t)}{t^{n_{j,\alpha,\beta}}}\exp\left(-B_{j,\alpha,\beta}(t)\frac{(x-y)^2}{t}\right).
\end{array}
\end{equation}
Next we extend these estimates to the whole expansion of the density.
For the fundamental solution $p$ of (\ref{hoerdiffstep3})
we show that on $(x,y)\not\in \delta $ and $y\not\in A_0$ the integral expression in
\begin{equation}\label{pexp1special}
\begin{array}{ll}
\frac{\partial^{|\alpha|}}{\partial x^{\alpha}} \frac{\partial^{|\beta|}}{\partial y^{\beta}}p(t,x;0,y):= \frac{\partial^{|\alpha|}}{\partial x^{\alpha}} \frac{\partial^{|\beta|}}{\partial y^{\beta}}N_0(t,x;0,y)\\
\\
+\int_0^{t}\int_{\Omega} \frac{\partial^{|\alpha|}}{\partial x^{\alpha}} \frac{\partial^{|\beta|}}{\partial y^{\beta}}N_0(t,x;\sigma,\xi)\phi(\sigma,\xi;0,y)d\sigma d\xi
\end{array}
\end{equation}
can be well-defined where $\phi$ is a recursively defined via (\ref{pexp2}). 
The essential step is to look at the set of points $y\in \Omega^m_0\subseteq \Omega$, where
\begin{equation}
D^{\gamma}_y\lambda_i(y)\neq 0 \mbox{ for all multiindices }|\gamma|\leq m,
\end{equation}
and where $m\geq |\alpha|+|\beta|$.

This is not trivial because a brute force differentiation expansion leads to singular summands  which are not integrable. It has to be shown that the integrals in the second summand of this representation make sense.
We use the special structure of $\phi$, and construct estimates following the successive approximations of $\phi$ in the construction above. First recall that for $y\not\in A_0$ that for the first approximation of $\phi$ we have
\begin{equation}
\begin{array}{ll}
L_0N_0:=\sum_{i,j=1}^n\left( a_{ij}(y)-a_{ij}(x)\right) \frac{\partial^2}{\partial x_i\partial x_j}N_0,
\end{array}
\end{equation}
and which extends naturally at points of degeneracy. Note that we have regular coefficient functions $a_{ij}\in C^{\infty}_b$ with
\begin{equation}
a_{ij}(x)-a_{ij}(y)=O(|x-y|).
\end{equation}
Now if we look at the first approximation we have to estimate the term
\begin{equation}\label{pexp1special2}
\begin{array}{ll}
\int_0^{t}\int_{\Omega} \frac{\partial^{|\alpha|}}{\partial x^{\alpha}} \frac{\partial^{|\beta|}}{\partial y^{\beta}}N_0(t,x;\sigma,\xi)\sum_{i,j=1}^n\left( a_{ij}(\xi)-a_{ij}(y)\right) \frac{\partial^2}{\partial \xi_i\partial \xi_j}N_0(\sigma,\xi;0,y)d\sigma d\xi
\end{array}
\end{equation}
As $y\in \Omega^m_0$, hence $D^{\alpha}_y\lambda_i(y)^{-1}\neq 0$ for all multiindices $|\alpha|\leq m$ we may 
do formal derivatives with respect to spatial $x$ and spatial $y$ variables such that we get a series of terms of the form
\begin{equation}\label{pexp1special2}
\begin{array}{ll}
\int_0^{t}\int_{\Omega} \frac{(x-\xi)^{\gamma}}{(t-\sigma)^p}N_0(t,x;\sigma,\xi)
\sum_{i,j=1}^n\left( a_{ij}(\xi)-a_{ij}(y)\right)
\frac{(\xi-y)^{\delta}}{\sigma^q} N_0(\sigma,\xi;0,y)\times\\
\\
\times O(\lambda_1(y),\cdots,\lambda_n(y))d\sigma d\xi,
\end{array}
\end{equation}
where $O$ in $O(\lambda_1(y),\cdots,\lambda_n(y))$ is a certain operator.
 The details of the operator are not of interest as long as $y\in \Omega^m_0$;
  we just note that $y\rightarrow O(\lambda_1(y),\cdots,\lambda_n(y))$ is a well-defined regular function with finite values. We note that we have no specific restrictions for the constellation of nonnegative integers $p,q,\gamma_i,\delta_j$. Hence for some smooth functions $f_{p,q,\gamma,\delta}$ we have essentially to estimate terms of the form 
\begin{equation}\label{pexp1special2}
\begin{array}{ll}
\int_0^{t}\int_{\Omega} \frac{(x-\xi)^{\gamma}}{(t-\sigma)^p}N_0(t,x;\sigma,\xi)
\frac{(\xi-y)^{\delta}}{\sigma^q} N_0(\sigma,\xi;0,y)f_{p,q,\gamma,\delta}(y)d\sigma d\xi\\
\\
=\int_0^{t/2}\int_{\Omega} \frac{(x-\xi)^{\gamma}}{(t-\sigma)^p}N_0(t,x;\sigma,\xi)
\frac{(\xi-y)^{\delta}}{\sigma^q} N_0(\sigma,\xi;0,y)f_{p,q,\gamma,\delta}(y)d\sigma d\xi\\
\\
+\int_{t/2}^t\int_{\Omega} \frac{(x-\xi)^{\gamma}}{(t-\sigma)^p}N_0(t,x;\sigma,\xi)
\frac{(\xi-y)^{\delta}}{\sigma^q} N_0(\sigma,\xi;0,y)f_{p,q,\gamma,\delta}(y)d\sigma d\xi
\end{array}
\end{equation}
Note that we assumed $x\neq y$. As an essential example we consider the second summand on the right side of the latter equation. We consider the integral first on the domain $\Omega_{\epsilon}=\Omega\setminus B_{\epsilon}(x)$, i.e., on the complement of the ball $B_{\epsilon}(x)$ of small radius $\epsilon>0$ around $x$. The only singularities are at $\sigma=t$. As we are considering points $y\in \Omega^m_0$ we have not to deal specifically with the $\lambda_i(y)$, the derivatives of their inverses and so on. We indicate that in the change of the subscript of the factor $f$. We may absorb also constant obtained by partial integration in these terms and speak of a 'proportional from'. Then for $p\geq 2$ one partial integration with respect to  time $\sigma$ leads to terms of the proportional form 
\begin{equation}
\begin{array}{ll}
\int_{t/2}^t\int_{\Omega_{\epsilon}} \frac{(x-\xi)^{\gamma+2}}{(t-\sigma)^{p+1}}N_0(t,x;\sigma,\xi)
\frac{(\xi-y)^{\delta}}{\sigma^q} N_0(\sigma,\xi;0,y)f_{p,q,\gamma+2,\delta}(y)d\sigma d\xi+\\
\\
\int_{t/2}^t\int_{\Omega_{\epsilon}} \frac{(x-\xi)^{\gamma+2}}{(t-\sigma)^{p-1}}N_0(t,x;\sigma,\xi)
\frac{\partial}{\partial \sigma}\left( \frac{(\xi-y)^{\delta}}{\sigma^q} N_0(\sigma,\xi;0,y)\right) f_{p,q,\gamma+2,\delta}(y)d\sigma d\xi
\end{array}
\end{equation}
where $\gamma+2$ is the multiindex with entries $\gamma_i+2$. For the second summand the derivative wit respect to $\sigma$ is in a domain well away from zero such that this part of the intergand is regular. Iterating this proxies we obtain after finitely many  or after $p$ steps terms of the form
\begin{equation}
\begin{array}{ll}
\int_{t/2}^t\int_{\Omega_{\epsilon}} \frac{(x-\xi)^{\gamma'}}{(t-\sigma)^{q}}N_0(t,x;\sigma,\xi)
\frac{(\xi-y)^{\delta}}{\sigma^q} N_0(\sigma,\xi;0,y)f_{p,q,\gamma+2,\delta}(y)d\sigma d\xi,
\end{array}
\end{equation} 
where $\gamma_i\geq 2q$. For these terms we have upper bounds of the form
\begin{equation}
\begin{array}{ll}
{\big |}Q(x-y)p_{\gamma'}\left(\frac{(x-y)}{t}\right)N_0(t,x;0,y){\big |}
\\
\\
\leq {\Big |}\sup_{z\in {\mathbb R}}\left( p_{\alpha}(z)
\exp\left(-\frac{1}{4}\sum_{i=1}^n\lambda^{-1}_{i}(y)z_i^2 \right)\right)\times\\
\\ 
\tilde{C}\frac{1}{\sqrt{4\pi \Pi_{i=1}^n\lambda_i(y) t}^n}\exp\left(-\frac{1}{2}\sum_{i=1}^n\lambda^{-1}_{i}(y)z_i^2 \right){\Big |}\\
\\
\leq C\frac{1}{\sqrt{4\pi \Pi_{i=1}^n\lambda_i(y) t}^n}\exp\left(-\frac{1}{2}\sum_{i=1}^n\lambda^{-1}_{i}(y)\frac{(x_i-y_i)^2}{t} \right)
\end{array}
\end{equation}
where $Q$ and $p_{\gamma'}$ are some polynomials. The estimates extend to the limits $\epsilon\downarrow 0$ and $y\in A_0$ off-diagonal. 

Next we consider additional time derivatives. Since the estimates are proved for the leading term it is sufficient to consider the integral terms. Start with the first order time derivative. We have 
\begin{equation}
\begin{array}{ll}
D_t\int_0^{t}\int_{\Omega} \frac{\partial^{|\alpha|}}{\partial x^{\alpha}} \frac{\partial^{|\beta|}}{\partial y^{\beta}}N_0(t,x;\sigma,\xi)\phi(\sigma,\xi;0,y)d\sigma d\xi,\\
\\
=
\int_0^{t}\int_{\Omega} D_t\frac{\partial^{|\alpha|}}{\partial x^{\alpha}} \frac{\partial^{|\beta|}}{\partial y^{\beta}}N_0(t,x;\sigma,\xi)\phi(\sigma,\xi;0,y)d\sigma d\xi\\
\\
+\int_{\Omega} \frac{\partial^{|\alpha|}}{\partial x^{\alpha}} \frac{\partial^{|\beta|}}{\partial y^{\beta}}N_0(t,x;\sigma,\xi)\phi(\sigma,\xi;0,y) d\xi{\Big |_{\sigma=t}},
\end{array}
\end{equation}
Applying higher order derivatives with respect to time and considering the expansion of $\phi$ we get a series where we observe that it is essential to estimate terms of the form
\begin{equation}
\int_0^{t}\int_{\Omega} D^m_t\frac{\partial^{|\alpha|}}{\partial x^{\alpha}} \frac{\partial^{|\gamma|}}{\partial \xi^{\gamma}}N_0(t,x;\sigma,\xi)\frac{\partial^{|\delta|}}{\partial \xi^{\delta}} \frac{\partial^{|\beta|}}{\partial y^{\beta}}N_0(\sigma,\xi;0,y)d\sigma d\xi\\
\end{equation}
for arbitrary multiindices $\alpha,\beta,\gamma,\delta$ and arbitrary nonnegative integers $m$. This leads to the same situation as in the case of pure multivariate spatial derivatives.

\end{proof}

\subsection{Existence of local densities in the case of the classical H\"{o}rmander condition}
If a drift is added to the diffusion equation above and the classical H\"{o}mander condition is satisfied at each point $x\in {\mathbb R}^n$, then we still have weak ellipticity in the sense that $(a_{ij})=\sigma\sigma^T\geq 0$, and can diagonalize the $\sigma\sigma^T$ such that we get still get a function $x\rightarrow \Lambda(x)=D(x)\sigma\sigma^T(x)D^T(x)$
where $\Lambda(x)=\mbox{diag}(\lambda_1(y),\cdots ,\lambda_n(x))$ is a diagonal matrix. We also still have $x\rightarrow \Lambda(x)\in C^{\infty}$. However, this time the rank of the matrix may be strictly less than $n$, even locally on any ball $B_{\epsilon}(x)$ (cf. also our example in section 2 of this paper). 
In this case infinitesimal rotations and infinitesimal translations induced by the drift, i.e, the first order coefficients of the operator, and the diffusion effect in subspaces work together in order to induce a diffusion effect on a whole ball. 
However, there is a difference in the estimates as information may be transported via a vector field while the diffusion effect is only via an interplay of rotations induced by the drift and diffusions which are effective in other space directions. This causes the additional polynomial factor $(1+|x|^{m_{j,\alpha,\beta}})$ in the Kusuoka Stroock estimates on a global scale.
In order to obtain a local expansion in this case and following our discussion in section 2 and the preceding part of this section
it is natural to consider for each $x\in {\mathbb R}^n$ and each ball $B_{\epsilon}(x)$ of radius $\epsilon$ around $x$ an $y\in B_{\epsilon}(x)$ such that for some $k\geq 1$ the equation
\begin{equation}\label{leadingtermeq}
\frac{\partial u}{\partial t}-\sum_{j=1}^k\lambda_{i_j}(y)\frac{\partial^2}{\partial x_{i_j}^2}+\sum_{j\in \{1,\cdots,n\}\setminus \{i_1,\cdots,i_k\}}b_j(x)\frac{\partial u}{\partial x_j}=0
\end{equation}
has $\lambda_{i_j}(y)>0$ for $1\leq j\leq k\geq 1$. We abbreviate ${\mathbb N}_n=\left\lbrace 1,\cdots,n\right\rbrace $. As the H\"{o}rmander condition is satisfied, we can choose $y$ such that in addition for each $j\in \{1,\cdots,n\}\setminus \{i_1,\cdots,i_k\}$  the coefficient $b_j$ depends effectively on the variable $x_{i_j}$ at $y$ for some $1\leq j\leq k$. This means that the flow ${\cal F}^{b_{i,1,\cdots,i_n}}_t$ associated to the equation
\begin{equation}
\frac{\partial u}{\partial t}+\sum_{j\in \{1,\cdots,n\}\setminus \{i_1,\cdots,i_k\}}b_j(x)\frac{\partial u}{\partial x_j}=0
\end{equation}
transports such that for small time $t>0$ the components of ${\cal F}^{b_{i,1,\cdots,i_n}}_ty$ 
in the complement of $\mbox{span}\left\lbrace e_{i_1},\cdots ,e_{i_k}\right\rbrace$ are dependent on some variable $x_{i_j}$ for $1\leq j\leq k$. This way the diffusion has a smoothing effect in all direction which are complementary to  $\mbox{span}\left\lbrace e_{i_1},\cdots ,e_{i_k}\right\rbrace$. Solving the equation (\ref{leadingtermeq}) for such an $y$ by a Trotter product iteration we get a natural leading term which is smoothing in all direction. We can then repeat the analysis of the preceding section in the case of the reduced H\"{o}rmander condition and do the same construction with this leading term in order to obtain a local construction of the full density in the case of the classical H\"{o}rmander condition.
 
Next we analyze this situation in more detail. First we have  
\begin{lem}\label{lem1}
Let ${\mathbb N}_n:= \left\lbrace 1,\cdots,n\right\rbrace$.
Assume that the H\"{o}rmander condition is satisfied and let $x\in {\mathbb R}^n$. Let
\begin{equation}
S_x:=\left\lbrace i\in {\mathbb N}_n|\lambda_i(x)>0 \right\rbrace, 
\end{equation} 
As the $\lambda_i$ are contiguous we consider a ball $B_{\epsilon}(x)$ of radius $\epsilon>0$ around $x$ such that for all $y\in B_{\epsilon}(x)$ we have 
\begin{equation}
S_y:=\left\lbrace i\in \left\lbrace 1,\cdots,n\right\rbrace |\lambda_i(y)>0 \right\rbrace\supseteq S_x.
\end{equation} 
and a finite family $\left\lbrace B^{i_1,\cdots,i_k}_{\epsilon}(x)|0\leq i_1\leq\cdots \leq i_k\leq n \right\rbrace$ of open sets in $B_{\epsilon}(x)$ with
\begin{equation}
B^{i_1,\cdots,i_k}_{\epsilon}(x)=\left\lbrace  y\in B_{\epsilon}(x)|S_y=\left\lbrace i_1,\cdots,i_k\right\rbrace \right\rbrace 
\end{equation}
such that the union of the closures of these open sets equals the set $B_{\epsilon}(x)$, and for some $y\in B_{\epsilon}^{i_1,\cdots,i_k}(x)$ we have
\begin{equation}\label{genhoerimpl}
S^{i_1,\cdots,i_k}_y\cup \left\lbrace j\in {\mathbb N}_n\setminus S^{i_1,\cdots,i_k}_y|\exists i\in S_y~\&~b_j(y)\lambda_i(y)\neq 0\right\rbrace=\left\lbrace 1,\cdots ,n\right\rbrace.
\end{equation}
\end{lem}
\begin{proof}
Since $\lambda_i$ is continuous for all $i\in \left\lbrace 1,\cdots ,n \right\rbrace$ there is an $\epsilon >0$ such that 
\begin{equation}
S_x\subseteq S_y
\end{equation}
for all $y\in B_{\epsilon}(x)$.
If (\ref{genhoerimpl}) this is not true for some $\left\lbrace i_1,\cdots,i_k\right\rbrace$, then there exists  $j\in {\mathbb N}_n\setminus S^{i_1,\cdots,i_k}_y$ such that for all $i\in S^{i_1,\cdots,i_k}_y$ and for all $y\in B^{i_1,\cdots,i_k}_{\epsilon}(x)$ we have
\begin{equation}
b_j(y)\lambda_i(y)=0.
\end{equation}
This implies that the Hoermander condition is not satisfied at $y$.
\end{proof}
Note that for the choice of $i,j$ such that $b_j(y)\lambda_i(y)\neq 0$ we may in addition assume that $b_j$ is effectively dependent on $y_i$, i.e., $y_i$ is not a critical point of  $z_i\rightarrow b_j(z)$.
Next the closure of $B^{i_1,\cdots,i_k}_{\epsilon}(x)$ is the closure of the finite union of the sets
\begin{equation}
B^{i_1,\cdots,i_k,j,i}_{\epsilon}(x):=\left\lbrace y|b_j(y)\lambda_i(y)>0\right\rbrace 
\end{equation}
over the indices $i\in S_y$ and $j\in {\mathbb N}_n\setminus S_y$. The set of points in
\begin{equation}
B^{i_1,\cdots,i_k,j,i}_{\epsilon}(x)\setminus \cup_{i\in S_y,~j\in {\mathbb N}_n\setminus S_y}B^{i_1,\cdots,i_k,j,i}_{\epsilon}(x)
\end{equation}
are of Lebesgues measure zero by the H\"{o}rmander condition and Sard's theorem. 

 Next we consider the rational points $Q$ and the set of $n$-tuples o rational points $Q^n$. Let $C_{Q}:=\left\lbrace B_{\epsilon_m}(x_m)|m\in {\mathbb N}\right\rbrace$ be an enumeration of all balls with all rational center points $x_m\in Q^n$ and all rational radius sizes $\epsilon_m>0$ chosen such that the statement of the preceding lemma holds. This set $C_Q$ is a covering of ${\mathbb R}^n$, and since ${\mathbb R}^n$ is locally compact, there is a locally finite subcover $C^1_Q$ of $C^0_Q$ which covers ${\mathbb R}^n$ as well. Now we may first construct densities via Trotter product formulas on finite intersections of sets 
\begin{equation}
\cap_{j\in {\mathbb N}_n\setminus \left\lbrace i_1,\cdots,i_k \right\rbrace, i=i(j) }B^{i_1,\cdots,i_k,j,i}_{\epsilon}(x).
\end{equation}
This can be done step by step splitting the operator. For each ball $B^{i_1,\cdots,i_k,j,i}_{\epsilon}(x)$ the reduced H\"{o}rmander condition holds on the subspace of ${\mathbb R}^n$
 built by the vectors $e_{i_m},~1\leq m\leq k$, where these are the standard basis vectors of the vector space ${\mathbb R}^n$ with the $i_m$th entry equal to $1$ and zero entries else.
  On this subsace we can use the result of the previous section and obtain a local construction of a density on $B^{i_1,\cdots,i_k,j,i}_{\epsilon}(x)\cap \mbox{span}\left\lbrace e_{i_1},\cdots ,e_{i_k}\right\rbrace$ which solves the equation
\begin{equation}
\frac{\partial p}{\partial t}-\sum_{j=1}^k\lambda_{i_j}\frac{\partial^2 p}{\partial x_{i_j}^2}=0
\end{equation}
on the closure of $B^{i_1,\cdots,i_k,j,i}_{\epsilon}(x)\cap \mbox{span}\left\lbrace e_{i_1},\cdots ,e_{i_k}\right\rbrace$. Then we may successively add drift terms as follows.If $k=n$ then the reduced H\"{o}H\"{o}rmander condition is satisfied and the existence of local densities with additional first order terms can be proved by  using the method of the previous section. If $k<n$, i.e., if the reduced H\"{o}rmander condition does not hold at $x$ and on $B^{i_1,\cdots,i_k,j,i}_{\epsilon}(x)$, then in a first step we consider $j\in {\mathbb N}_n\setminus \left\lbrace i_1,\cdots ,i_k\right\rbrace$, where for some $i\in \left\lbrace i_1,\cdots ,i_k\right\rbrace $ we have our witness of a rotational effect, i.e.,
\begin{equation}
b_j(y)\lambda_i(y)\neq 0
\end{equation}
If we look at the situation from a Trotter product perspective where we apply a first order vector field operator related to the operator term
\begin{equation}
b_j\frac{\partial }{\partial x_j}
\end{equation}
and a reduced H\"{o}rmander diffusion on $\mbox{span}\left\lbrace e_{i_1},\cdots ,e_{i_k}\right\rbrace$ related to the operator
\begin{equation}
\frac{\partial}{\partial t}-\sum_{j=1}^k\lambda_{i_j}\frac{\partial^2}{\partial x_{i_j}^2}
\end{equation}
in small time steps and take a Trotter product limit, then we observe the phenomenological essence of the H\"{o}rmander result. The first order terms cause small rotations such that the diffusion which is active on a $k$ dimensional subspace smoothes the value function in the direction $j\not\in \left\lbrace i_1,\cdots ,i_k\right\rbrace$. 
This means that in a small time step the vector field transports in some direction $e_j$ perpendicular to $\mbox{span}\left\lbrace e_{i_1},\cdots ,e_{i_k}\right\rbrace$, and the diffusion can act on the rotated space and has a smoothing effect which now includes the direction of $e_j$. This phenomenon is preserved in the Trotter product limit. 
This way a density for the equation
\begin{equation}
\frac{\partial p}{\partial t}-\sum_{j=1}^k\lambda_{i_j}\frac{\partial^2}{\partial x_{i_j}^2}+b_j\frac{\partial p}{\partial x_j}=0
\end{equation}
is obtained in a Trotter product limit. This procedure can be repeated in $n-k$ steps in the obvious way such that a density of the operator with the full drift term is obtained.
We may start this procedure with the operator where the reduced H\"{o}rmander condition is satisfied, o course, i.e., with the equation
\begin{equation}
\frac{\partial p}{\partial t}-\sum_{j=1}^k\lambda_{i_j}\frac{\partial^2}{\partial x_{i_j}^2}+\sum_{j=1}^kb_{i_j}\frac{\partial p}{\partial x_j}=0,
\end{equation}   
and apply the construction of the previous section.

Next we formalize these observations. First we have to construct the leading term based on the property (\ref{genhoerimpl}). We shall define it by a Trotter product formula defined via iteration of a pure diffusion and a first order vector field operator. Firs for first order differentail equations we observe 
\begin{prop}\label{prop}
 Assume that $\mu_i\in C^1_b\left([0,\infty)\times {\mathbb R}^n \right) $ and let $g\in C^1_b\left([0,\infty)\times {\mathbb R}^n \right) $. Then there exists a smooth global flow ${\cal F}^t$ generated by the vector field 
\begin{equation}\label{vec1*}
\sum_{i=1}^n\mu_i(x)\frac{\partial}{\partial x_i}
\end{equation}
 on $[0,\infty)\times {\mathbb R}^{n}$ such that the first order equation problem
\begin{equation}\label{vecf}
\begin{array}{ll}
\frac{\partial u}{\partial t}=\sum_{i=1}^n\mu_i(x^{n})\frac{\partial}{\partial x_i}u+g(x^{n}),\\
\\
~~u(0,x^{n})=f(x^{n}),
\end{array}
\end{equation}
has the solution
\begin{equation}\label{sol}
u(t,x^{n})=f\left({\cal F}^t x^{n}\right)+\int_0^tg({\cal F}^{t-s}x^{n})ds. 
\end{equation}

\end{prop}

\begin{proof}
Note that the variables $x^d$ are fixed and serve as parameters. Consider the characteristic form
\begin{equation}
\chi_L(z,\xi)=\xi_0-\sum_{i=1}^n\mu_i\xi_i
\end{equation}
of the operator $L\equiv\frac{\partial}{\partial t}-\sum_{i=1}^n\mu_i \frac{\partial}{\partial x_i}$, where $\xi = (\xi_{0},\xi_{1},\ldots,\xi_{n})$. The surface $S:=\left\lbrace t=0 \right\rbrace$ has a constant normal vector $(1,0,\cdots ,0)$, hence is non-characteristic for the surface $S$, i.e. at any point $z=(t,x)$ we have 
\begin{equation}
(1,0,\cdots ,0)\notin \mbox{char}_z(L):=\left\lbrace \xi\neq 0|\xi_0-\sum_{i=1}^n \mu_i\xi_i=0 \right\rbrace \text{.}
\end{equation}
Hence, basic PDE-theory tells us that the first order Cauchy problem has a unique local solution in a sufficiently small neighborhood of the surface $S$ and is given in the form of solutions of associated ODEs along its characteristic curves. This leads to a solution up to a time $T_1$. Next we may iterate the argument in time. Assume that this does not lead to a global solution but to a limit $T_{\infty}>0$. Then on the time horizon $\left[0,T_{\infty}\right]$ we have a classical solution. Moreover the solution has a representation on this horizon as a family of ODE-solutions along characteristic curves, and where the assumptions on the coefficients	  imply that this family of solutions is uniformly bounded up to time $T_{\infty}$. Hence we may apply the first order PDE argument above again for the Cauchy problem with initial data $S_{T_{\infty}}:=\left\lbrace t=T_{\infty} \right\rbrace$ and extend the solution beyond the horizon $\left[0,T_{\infty}\right]$.     Hence there is a unique global solution. 
For each $x_0^{n}\in {\mathbb R}^{n}$ the flow ${\cal F}_t$ of the vector field $\sum_i \mu_i\frac{\partial}{\partial x_i}$ defines a characteristic curve $x_0^{n}(t):={\cal F}_tx_0^{n}$. Note that
\begin{equation}
{\cal F}^t x^{n} 
\end{equation}
is a solution of the homogeneous Cauchy problem
\begin{equation}\label{vecfhom}
\begin{array}{ll}
\frac{\partial u}{\partial t}=\sum_{i=1}^n\mu_i(x)\frac{\partial}{\partial x_i}u,\\
\\
~~u(0,x^n)=f(x^{n}),
\end{array}
\end{equation} 
and then the form of the solution (\ref{sol}) of the inhomogeneous equation follows from Duhamel's principle.
\end{proof}
Next we have a local Trotter product result, i.e., we can prove a local form of the main part of theorem 1.3 for problems without first order terms where a reduced H\"{o}rmander condition holds, i.e., we prove
\begin{lem}
Consider a class of Cauchy problems 
\begin{equation}
	\label{hoer23}
	\left\lbrace \begin{array}{ll}
		\frac{\partial q}{\partial t}=\sum_{i=1}^mV_i^2q+V_0q\\
		\\
		q(0,x;y)=\delta_y(x),
	\end{array}\right.
\end{equation}
which is equivalent to the diffusion problem
\begin{equation}
\label{hoerdiff03}
	\left\lbrace \begin{array}{ll}
		0=\frac{\partial q}{\partial t}-\sum_{i,j=1}^na_{ij}\frac{\partial^2}{\partial x_i\partial x_j}q +\sum_{i=1}^nb_i\frac{\partial q}{\partial x_i}=Lq\\
		\\
		q(0,x;y)=\delta_y(x),
	\end{array}\right.
\end{equation}
and assume that the classical H\"{o}rmander condition is satisfied for all $y\in B_{\epsilon}(x)$, i.e.,
\begin{equation}
H_x\left[V_0,V_1,\cdots ,V_n \right] ={\mathbb R}^n \mbox{ for all }x\in B_{\epsilon}(x).
\end{equation}
Then there exists a local density $p$ which is a classical solution of 
(\ref{hoerdiff0}, i.e. for every $x\in {\mathbb R}^n$ there exists a ball $B_{\epsilon}(x)$ and a function
\begin{equation}
p:\left(0,\infty\right)\times B_{\epsilon}(x)\times B_{\epsilon}(x)\rightarrow {\mathbb R},
\end{equation}
which satisfies (\ref{hoerdiff03}), and has the distributional limit $\delta_y(x)$ as $t\downarrow 0$. 
  Furthermore, the Kusuoka Stroock estimates are satisfied locally.
\end{lem}

\begin{proof}
For each $x_0\in {\mathbb R}^n$ we consider some $y\in B_{\epsilon}(x_0)$ such that for some set $\left\lbrace i_1,\cdots,i_k\right\rbrace\subseteq {\mathbb N}_n$ for some $k\geq 1$ we have
\begin{equation}
\lambda_{i_j}(y)>0~\mbox{ for }~1\leq j\leq k,
\end{equation}
and where for all $j\in {\mathbb N}_n\setminus \left\lbrace i_1,\cdots,i_k\right\rbrace$ there is $i_j\in \left\lbrace i_1,\cdots,i_k\right\rbrace$ such that
\begin{equation}
b_j(y)\lambda_{i_j}(y)\neq 0.
\end{equation}
A leading term can now be constructed as a solution of the equation
\begin{equation}\label{initialdiff}
L^y\equiv\frac{\partial u}{\partial t}-\sum_{j=1}^k\lambda_{i_j}(y)\frac{\partial^2 u}{\partial x_{i_j}^2}+\sum_{j\in {\mathbb N}_n\setminus \{i_1,\cdots ,i_k\}}b_j(x)\frac{\partial u}{\partial x_j}=0.
\end{equation}
For each such $y$ in $B_{\epsilon}(x_0)$ we consider the closure of the open neighbourhood $U(y)$ where for all
$z\in U(y)$ we have
\begin{equation}
\left\lbrace i|\lambda_{i}(z)>0\right\rbrace \supseteq 
\left\lbrace i_1,\cdots,i_k\right\rbrace. 
\end{equation}
from the previous considerations we know that $B_{\epsilon}(x_0)$ is covered by finitely many of these type of closed neighborhoods $U(y^1),\cdots ,U(y^p)$. Note that the boundaries of these sets are critical sets of regular functions and therefore of Lebesgues measure zero.
For $1\leq q\leq p$ we solve equation (\ref{initialdiff}) with diffusion coefficients evaluated at $y=y^q$  by operator splitting and iterations of solutions of the pure diffusion equation
\begin{equation}\label{reddiff}
\frac{\partial u}{\partial t}-\sum_{j=1}^k\lambda_{i_j}(y^q)\frac{\partial^2 u}{\partial x_{i_j}^2}=0,
\end{equation}
with Dirac data, and solutions of  the vector field equation
\begin{equation}\label{flowi}
\frac{\partial u}{\partial t}+\sum_{j\in {\mathbb N}_n\setminus \{i_1,\cdots ,i_k\}}b_j(x)\frac{\partial u}{\partial x_j}=0,
\end{equation} 
which transport the data. Next, after possible relabeling of coordinates we may assume
\begin{equation}\label{conv}
\left\lbrace i_1,\cdots,i_k\right\rbrace=\left\lbrace 1,\cdots,k \right\rbrace 
\end{equation}
for convenience of notation in the following.
As we remarked we may assume that $y^q$ is outside some critical set such that for all $j\in {\mathbb N}_n\setminus \left\lbrace 1,\cdots,k  \right\rbrace$ we find an $i_j\in \left\lbrace 1,\cdots,k\right\rbrace$ such that $z_{i_j}\rightarrow b_j(z)$ is not critical at $z_{i_j}=y_{i_j}$. This means that the flow ${\cal F}^{{\mathbb N}^c_k}_t$ (the superscript ${\mathbb N}^c_k$ indicates that the flow transports the coordinates in ${\mathbb N}_n\setminus \left\lbrace 1,\cdots,k\right\rbrace$ associated to (\ref{flowi}) with (\ref{conv}) has the effect that for $p\in \left\lbrace k+1,\cdots,n\right\rbrace$  the $z_p$ of
\begin{equation}
z(t):={\cal F}^{{\mathbb N}^c_k}_t\left(y_1,\cdots y_k,y_{k+1},\cdots,y_n \right)^T 
\end{equation}
depends on some $x_i$ for $i\in \left\lbrace 1,\cdots,k \right\rbrace$ for some small $t>0$
such that we have a smoothing effect of the diffusion in all spatial directions in $U(y^q)$. 
Then a local contraction argument shows that for initial data $f$ and with the abbreviation
\begin{equation}
L^q_0\equiv\frac{\partial u}{\partial t}-\sum_{j=1}^k\lambda_{i_j}(y^q)\frac{\partial^2 u}{\partial x_{i_j}^2}
\end{equation}
the problem for the leading term in (\ref{initialdiff}) with  $y=y^q$ has a solution
\begin{equation}
u(t,x)=\lim_{m\uparrow \infty}
\left(\exp\left(\frac{t}{m}L^q_0\right) 
{\cal F}^{{\mathbb N}^c_k}_{\frac{t}{m}}\right) ^m  f(x)
\end{equation}
on $U(y^q)$. The associated density $p_{y,U(y^q)}$ is the leading term of the local expansion on $U(y^q)$ of the density $p_{U(y^q)}$ of the equation 
\begin{equation}\label{initialdiff}
\frac{\partial u}{\partial t}-\sum_{j=1}^k\lambda_{i_j}(y)\frac{\partial^2 u}{\partial x_{i_j}^2}+\sum_{j\in {\mathbb N}_n\setminus \{i_1,\cdots ,i_k\}}b_j(x)\frac{\partial u}{\partial x_j}=0.
\end{equation}
with variable second order coefficients. 
We then get the expansion
\begin{equation}\label{pexp1**}
\begin{array}{ll}
p_{U(y^q)}(t,x;s,y):=p_{y,U(y^q)}(t,x;s,y)\\
\\
+\int_s^{\tau}\int_{U^q}p_{y,U(y^q)}(t,x;\sigma,\xi)\phi(\sigma,\xi;s,y)d\sigma d\xi,
\end{array}
\end{equation}
where $\phi$ is a recursively defined function, i.e.,
\begin{equation}\label{pexp2**}
\phi(t,x;s,y)=\sum_{m=1}^{\infty}(L^yp_{y,U(y^q)})_m(t,x;s,y),
\end{equation}
along with the recursion
\begin{equation}\label{pexp3**}
\begin{array}{ll}
(L^yp_{y,U(y^q)})_1(t,x;s,y)=L^yp_{y,U(y^q)}(t,x;s,y)\\
\\
(L^yp_{y,U(y^q)})_{m+1}(t ,x):=\int_s^t\int_{\Omega}\left( L^yp_{y,U(y^q)}(t,x;\sigma,\xi)\right)_m \times\\
\\
\times L^y p_{y,U(y^q)}(\sigma,\xi;s,y)d\sigma d\xi.
\end{array}
\end{equation} 
Concerning the Kusuoka Stroock type estimates we consider the definition of the local density by
\begin{equation}
\lim_{m\uparrow \infty}
\left(\exp\left(\frac{t}{m}L^q_0\right) 
{\cal F}^{{\mathbb N}^c_k}_{\frac{t}{m}}\right) ^m
\end{equation}
and prove the estimate for
\begin{equation}
\left(\exp\left(\frac{t}{m}L^q_0\right) 
{\cal F}^{{\mathbb N}^c_k}_{\frac{t}{m}}\right)
\end{equation}
for small time $t/m$ first and the consider the limit $m\uparrow \infty$. The spatial derivatives associated with the flow lead to an additional polynomial factor on $x$ which cannot be absorbed by the diffusion since the diffusion may be effective only in certain directions. For the same reason the time dependence of the estimation constants cannot be absorbed by the diffusion,  and we find some constants $A^q_{j,\alpha,\beta}(t)$, $B^q_{j,\alpha,\beta}(t)$ and some integers $m^q_{j,\alpha,\beta}, n^q_{j,\alpha,\beta}$ such that
\begin{equation}\label{pxestN**}
\begin{array}{ll}
{\Bigg |}\frac{\partial^j}{\partial t^j} \frac{\partial^{|\alpha|}}{\partial x^{\alpha}} \frac{\partial^{|\beta|}}{\partial y^{\beta}}p_{y,U(y^q)}{\Bigg |}
\leq \frac{A^q_{j,\alpha,\beta}(t)(1+x)^{m^q_{j,\alpha,\beta}}}{t^{n^q_{j,\alpha,\beta}}}\exp\left(-B^q_{j,\alpha,\beta}(t)\frac{(x-y)^2}{t}\right).
\end{array}
\end{equation}
Since the estimates are local, we could absorb the term  $(1+x)^{m^q_{j,\alpha,\beta}}$ by a constant depending on the local domain $U(y^q)$. Similar for the time-dependence However, can now observe why these type of constants appear naturally in the global estimates. A similar estimate can then be obtained for the density $p_{U(y^q)}$ using a similar argument as in the case of the reduced H\"{o}rmander diffusion in the preceding section.
  \end{proof}

\subsection{Existence of global densities}
If the full H\"{o}rmander condition holds, then the set 
\begin{equation}
\begin{array}{ll}
R:={\Big \{}y\in {\mathbb R}^n|\forall p\in\left\lbrace 1,\cdots ,k \right\rbrace \lambda_{i_p}(y)>0~\&~\forall j\in {\mathbb N}_n\setminus \left\lbrace i_1,\cdots,i_k\right\rbrace \\
\\
\exists i\in \left\lbrace i_1,\cdots ,i_k\right\rbrace:~y\not\in C(i,b_j) {\Big \}},
\end{array}
\end{equation}
where
\begin{equation}
C(i,b_j):=\left\lbrace y|~y~\mbox{is critical point of }z_i\rightarrow b_j(z)\right\rbrace 
\end{equation}
is dense in ${\mathbb R}^n$. This means that there is an open cover of balls $B_{\epsilon}(y^j)$ where in each open ball we have an expansion of the density on the domain $(0,T)\times B_{\epsilon}(y^j)\times (0,T)\times B_{\epsilon}(y_j)$. 
 In order to globalize the results spatially we observe that we have a locally finite cover 
\begin{equation}\label{cover}
\cup_{i\in {\mathbb N}}B_{\epsilon}(z_i)={\mathbb R}^n
\end{equation}
for some $z_i\in {\mathbb R}^n$, where on each domain $(0,T)\times B_{\epsilon_i}(z_i)\times (0,T)\times B_{\epsilon_i}(z_i)$ a local density $p_{B_{\epsilon_i}(z_i)}$ exists.
\begin{rem}
In order to do so we may first consider the countable set of all center points $z_i\in Q^n$, where $Q$ denotes the set of rational numbers. Around each $z_i$ we find a ball $B_{\epsilon_i}(z_i)$ where we have a classical density on the domain $(0,T)\times B_{\epsilon_i}(z_i)\times (0,T)\times B_{\epsilon_i}(z_i)$. As the space ${\mathbb R}^n$ is locally compact, we finite a locally finite subcover of ${\mathbb R}^n$
\end{rem} 
For some arbitrary $x\in {\mathbb R}^n$ we may assume (upon renumeration if necessary) that
\begin{equation}
x\in B_{\epsilon}(z_0).
\end{equation}
According to the previous section the density $p_{B_{\epsilon}(z_0}$ on $\left(0,T\right)\times B_{\epsilon}(z_0)\times (0,T)\times B_{\epsilon}(z_0)$ exists. 
Furthermore, we may assume (again upon renumeration if necessary) that for all $i\in {\mathbb N}$
\begin{equation}
B_{\epsilon}(z_i)\cap \left( \cup_{j\in \left\lbrace 1,\cdots ,i-1\right\rbrace  }B_{\epsilon}(z_j)\right)\neq \oslash.
\end{equation}
We set up a scheme which produces global densities $p$ on $\left(0,\infty\right)\times {\mathbb R}^n \times \left(0,\infty \right)\times {\mathbb R}^n$ starting from $p_{B_{\epsilon}(z_0)}$. We do this inductively where we use a density property.
For $i\geq 2$ we denote
\begin{equation}
U_{i-1}:=\cup_{j\in \left\lbrace 1,\cdots ,i-1\right\rbrace  }B_{\epsilon}(z_j)
\end{equation}
and assume that a density $p_{U_{i-1}}$ on $\left(0,\infty\right)\times U_{i-1} \times \left(0,\infty \right)\times U_{i-1}$ has been constructed. 
For $U_i=U_{i-1}\cup B_{\epsilon}(z_i)$ we consider a partion of unity $\phi_1,\phi_2$ subordinate to $U_{i-1},B_{\epsilon}(z_i)$ (the existence follows from standard results such  as Proposition 1.2 in \cite{KHyp}) and define a densities 
$p_{U^k_i},~k\geq 0$ on $\left(0,\infty\right)\times U_i \times \left(0,\infty \right)\times U_i$. We start with
\begin{equation}
p_{U_i}(t,x;s,y)=\int_{U_i}\left( \phi_1(z)p_{U_{i-1}}(t,x;\sigma,z)\right) \left( \phi_2(z)p_{B_{\epsilon}(z_i)}(\sigma,z;s,y)\right)c_{U_i}dsdy,
\end{equation}
where $c_{U_i}>0$ is a normalisation constant.
As we have a locally finite cover of ${\mathbb R}^n$ the limit $i\uparrow \infty$ leads to a global density $p$, where we may define
\begin{equation}
p(t,x;s,y)=\lim_{i\uparrow \infty}p_{U_i}(t,x;s,y)
\end{equation}
on $\left(0,\infty\right)\times {\mathbb R}^n \times \left(0,\infty \right)\times {\mathbb R}^n$.

\section{Extension of the proof to the case of the weak H\"{o}rmander condition}

Assume that the bounded coefficients  $a_{ij},b_k,~1\leq i,j,k\leq n$ satisfy a weak H\"{o}rmander condition with the dense set $D\subset {\mathbb R}^n$. These coefficients can be approximated by a sequence of smooth coefficient functions $a^m_{ij},b^m_k,~1\leq i,j,k\leq n,~m\in {\mathbb N}$ which satisfy the classical H\"{o}rmander condition and converge pointwise to the coefficients $a_{ij},b_k,~1\leq i,j,k\leq n$. The simple idea to prove this is as follows (cf. also \cite{FK}). We do not need to repeat this here, but let is consider the main idea.
Start with a family of smooth bounded coefficients $a^0_{ij},b^0_k,~1\leq i,j,k\leq n$ which satisfy the classical H\"{o}rmander condition and consider a countable dense subset $D_0\subset D$ with $\overline{D_0}={\mathbb R}^n$, where $D_0=\left\lbrace z_{p/q}|q\in Q \right\rbrace$, and where $Q$ denotes the set of rational numbers.  For fixed $q>1$ consider the countable subset $D^q_0:=\left\lbrace z_{p/q}\in D_0|p\in {\mathbb N}~\&~q~\mbox{fixed}\right\rbrace$. Then for this set using an appropriate partion of unity we easily construct smooth coefficients $a^q_{ij},b^q_k,~1\leq i,j,k\leq n$ which satisfy the H\"{o}rmander condition on $D^q_0$ and equal $a_{ij},b_k,~1\leq i,j,k\leq n$ on $D^q_0$. this can be done inductively where in the limit we have coefficients which satisfy the classical H\"{o}rmander condition on $D_0=\cup_{q\geq 1}D^q_0$ and are globally Lipschitz. We note 

\begin{lem}
Given second order coefficients  $a_{ij},~1\leq i,j\leq n$ and $b_k,~1\leq k\leq n$ which satisfy the H\"{o}rmander condition in the weak sense, i.e., in the sense that there exists a dense set $D$ of ${\mathbb R}^n$ such that the classical H\"{o}rmander condition is satisfied on $D$, and the coefficients are globally Lipschitz on the complementary set, then there exists a sequence of functions $a_{ij}^m\in C^{\infty}$ and $b^m_i\in C^{\infty}$ with
\begin{equation}
\lim_{m\uparrow \infty}a_{ij}^m(x)=a_{ij}(x),
\end{equation}
for all $1\leq i,j\leq n$, and
\begin{equation}
\lim_{m\uparrow \infty}b_k^m(x)=b_k(x)
\end{equation}
for all $1\leq k\leq n$, and such that the classical H\"{o}rmander condition is satisfied for all coefficient data $a_{ij}^m(x), b_k^m(x)$
for all $1\leq i,j,k\leq n$ and $m\in {\mathbb N}$, i.e.,
\begin{equation}
H^x\left[a_{ij}^m(x), b_k^m(x),1\leq i,j,k\leq n \right]={\mathbb R}^n 
\end{equation}for all $x\in {\mathbb R}^n$

\end{lem}

Using this lemma we first recall that a unction $p:(0,T)\times \Omega\times \Omega\rightarrow {\mathbb R}$ for some domain $\Omega\subset {\mathbb R}^n$ is a density or a fundamental solution of the second order equation
\begin{equation}
Lu\equiv \frac{\partial u}{\partial t}
-\sum_{i,j=1}^na_{ij}\frac{\partial^2 u}{\partial x_i\partial x_j}+\sum_{j=1}^nb_j\frac{\partial p}{\partial x_j}=0,
\end{equation}
if this function $(t-s,x;y)\rightarrow p(t-s,x;y)$ satisfies $Lu=0$ for fixed $s,y$ and $t>s$, and of for all bounded data $f\in C^{\infty}\left(\overline{\Omega}\right) $ we have
\begin{equation}
\lim_{t\downarrow s}\int_{\Omega}f(y)p(t-s,x,y)dy=f(x).
\end{equation}
In addition we may require that the spatial integral over $\Omega$ is normed by $1$. 
 We remark that we this definition is in the time-homogeneous form, and that it is the classical definition which avoids distributions, and which is obviously equivalent to the distributional formaltion which requires that the initial data are given by the Dirac delta distribution $\delta(x-y)$ for all $y\in {\mathbb R}^n$. Furthermore, this classical definition is sometimes given such that the all quantifier ranges over all bounded continuous functions which is also equivalent. 
Next we consider a sequence 
$a_{ij}^m\in C^{\infty}$ and $b^m_i\in C^{\infty}$ which satisfies the classical 
\begin{equation}
\lim_{m\uparrow \infty}a_{ij}^m(x)=a_{ij}(x),~\lim_{m\uparrow \infty}b_{i}^m(x)=b_i(x)
\end{equation}
or all $1\leq i,j,k\leq n$ where the coefficients $a_{ij},b_i,1\leq i,j,k\leq n$ satisfy the the weak H\"{o}rmander condition.
Next using the observation of the preceding section it is sufficient to show the existence of a classical limit $p=\lim_{m\uparrow }p^m$ locally, i.e., that for each $x\in {\mathbb R}^n$ there is $B_{\epsilon}(x)$ such that the density $p_m$ which solves
\begin{equation}\label{equationm}
L^mu:=\equiv \frac{\partial u^m}{\partial t}
-\sum_{i,j=1}^na^m_{ij}\frac{\partial^2 u^m}{\partial x_i\partial x_j}+\sum_{j=1}^nb^m_j\frac{\partial u^m}{\partial x_j}=0
\end{equation}
on $(0,T)\times B_{\epsilon}(x)$ for some $T>0$, and along with $u^m(0,x)=f(x)$ for given $f\in C^{\infty}\left( B_{\epsilon}(x)\right)$. Now for a solution $u^m$ of (\ref{equationm}) we have
\begin{equation}
Lu^m=\sum_{i,j=1}^n\left(a_{ij}-a^m_{ij} \right)\frac{\partial^2 u^m}{\partial x_i\partial x_j}+ \sum_{k=1}^n\left(b^m_{k}-b_{k} \right)\frac{\partial u^m}{\partial x_k}
\end{equation}
Locally, we the Kusuoka-Stroock type estimate have a Gaussian type upper bound and the Lipschitz continuity of the coefficients $a_{i,j},a^m_{ij}$ and $b_k,b^mk$ implies that the right side is uniformly bounded and goes to zero as $m\uparrow \infty$.  
Hence,
\begin{equation}
\lim_{m\uparrow \infty}\sum_{i,j=1}^n\left(a_{ij}-a^m_{ij} \right)\frac{\partial^2 u^m}{\partial x_i\partial x_j}+ \lim_{m\uparrow \infty}\sum_{k=1}^n\left(b^m_{k}-b_{k} \right)\frac{\partial u^m}{\partial x_k}=0
\end{equation}
for all $f$ such that $p=\lim_{m\uparrow \infty}p^m$ exists in a classical sense. We may then use the argument of the preceding section in order to prove this classical limit globally in space and time.

Finally we note that our construction leads straightforwardly to local adjoints mentioned in \cite{KHyp}, where we have used the local adjoint in order to generalise a global scheme for a basic equation of fluid mechanics.

\end{document}